\documentclass[a4paper,12pt,leqno]{amsart}
\usepackage{verbatim}
\usepackage{enumitem}
\usepackage{hyperref}
\usepackage{amssymb}
\usepackage[all]{xy}
\usepackage{tikz}
\usetikzlibrary{positioning,arrows,cd}

\newcommand{\R}{\mathbb{R}}
\newcommand{\N}{\mathbb{N}}

\newcommand{\calm}{{\mathcal M}}

\numberwithin{equation}{section}

\newtheorem{theorem}[equation]{Theorem}
\newtheorem{lemma}[equation]{Lemma}
\newtheorem{corollary}[equation]{Corollary}
\newtheorem{proposition}[equation]{Proposition}

\theoremstyle{definition}
 \newtheorem{definition}[equation]{Definition}
\newtheorem{example}[equation]{Example}
\theoremstyle{remark}
\newtheorem{remark}[equation]{Remark}

\usepackage{colortbl}
\DefineNamedColor{named}{RoyalBlue}     {cmyk}{1,0.50,0,0}
\DefineNamedColor{named}{BrickRed}      {cmyk}{0,0.89,0.94,0.28}

\begin{document}

\title{The achievement set of generalized multigeometric sequences}

\author[D. Karvatskyi]{Dmytro Karvatskyi}
\address{Department of Dynamical Systems and Fractal Analysis, Institute of Mathematics of NAS of Ukraine, 01024-Kyiv, Ukraine}
\email{karvatsky@imath.kiev.ua}

\author[A. Murillo]{Aniceto Murillo}
\address{Departamento de \'Algebra, Geometr\'{\i}a y Topolog\'{\i}a, Universidad de M\'alaga,
Campus de Teatinos, s/n, 29071-M\'alaga, Spain}
\email{aniceto@uma.es}

\author[A. Viruel]{Antonio Viruel}
\address{Departamento de \'Algebra, Geometr\'{\i}a y Topolog\'{\i}a, Universidad de M\'alaga,
Campus de Teatinos, s/n, 29071-M\'alaga, Spain}
\email{viruel@uma.es}

\thanks{The first author was supported by the University of M{\'a}laga grant D.3-2023 for research stays of renowned  Ukrainian scientists.
The second and third authors were partially supported by the MINECO grant PID2020-118753GB-I00 of the AEI of the Spanish Government.}


\begin{abstract}
 We study the topology  of all possible subsums of the {\em generalized multigeometric series}
 $k_1f(x)+k_2f(x)+\dots+k_mf(x)+\dots + k_1f(x^n)+\dots+k_mf(x^n)+\dots,$
where $k_1, k_2, \dots, k_m$ are fixed positive real numbers and $f$ runs along a certain class of non-negative functions on the unit interval. We detect particular regions of this interval for which this achievement set is, respectively, a compact interval, a Cantor set and a Cantorval.

\end{abstract}

\maketitle
\section{Introduction}

Let $E(z_n)$ denote the set of all possible sums of elements of the sequence $\{z_n\}$, or equivalently, all possible subsums of the series $\sum_{n\ge 1} z_n$. That is,
$$
E(z_n)=\left\{\sum_{n=1}^\infty c_nz_n,\,c_n\in\{0,1\}\right\}=\left\{\sum_{n\in A}z_n,\,A\in\N\right\}.
$$
Also known as the {\em achievement set} of the sequence  $\{z_n\}$ \cite{jo,zbMATH07383831,zbMATH07680677}, it was first considered by Kakeya \cite{Kakeya} who conjectured that, for a convergent positive series, this set is either a finite union of compact intervals or a Cantor set. This conjecture was refuted in \cite{Guthrie-Nymann} by means of the following result whose proof was completed in \cite{Nymann-Saenz}.

\begin{theorem}
\label{GN-Theorem}
The achievement set of a summable positive sequence is either:
\begin{enumerate}[label={\rm (\roman{*})}]
\item\label{thm:GN.1} a finite union of closed bounded intervals,
\item\label{thm:GN.2} homeomorphic to the Cantor set, or
\item\label{thm:GN.3} homeomorphic to a Cantorval.
\end{enumerate}
\end{theorem}

Recall that a (symmetric) {\em Cantorval} can be formally defined as a nonempty compact real subspace  which is the closure of its interior and the endpoints of any nontrivial component of this set are accumulation points of trivial components. All Cantorvals are homeomorphic to $[0,1]\backslash\cup_{n\ge 1}B_{2n}$ in which $B_n$ is the union of the $2^{n-1}$ open intervals which are eliminated at the $n$th stage of the construction of the Cantor set, which can be seen as $E(\frac{2}{3^n})$. In particular, any Cantorval is homeomorphic to the {\em Guthrie-Nymann Cantorval} given by $E(z_n)$, with $z_{2n}=2/4^n$ and $z_{2n-1}=3/4^n$, the one originally considered in \ref{thm:GN.3} of the above Theorem. Cantorvals also appear as attractors associated to some iterated function systems \cite{Banakiewicz}, and an analogous result to Theorem \ref{GN-Theorem} holds to describe the topology of the algebraic difference of certain Cantor sets \cite{achle,Mendes,no}.  

In this paper we consider a class of positive functions $f$ defined on the interval (see next section) and study the topological behavior, depending on $x$,  of  the set all possible subsums of the series
\begin{equation}\label{gms}
 k_1f(x)+\dots+k_mf(x)+k_1f(x^2)+ \dots + k_mf(x^2)+\dots + k_1f(x^n)+\dots+k_mf(x^n)+\dots,
\end{equation}
where $k_1,\dots, k_m $ are fixed positive scalars. Following the nomenclature in \cite{theFerdinands} we call this a {\em generalized multigeometric series}. We show that, whenever $x$ varies along some particular regions, the achievement set of the associated sequence is a compact interval (Theorem \ref{thm:intervals} and Corollary \ref{cor:contains_a_interval}), a Cantor set (Theorem \ref{thm:cantor}), or a Cantorval (Corollary \ref{cor:Cantorval_ferdinands}).

This extends the main results of \cite{dmytros} where $f(x)=\sin(x)$, and those in \cite{BFS} and \cite{theFerdinands}  where $f(x)=x$, i.e., the considered series is the multigeometric,
\begin{equation}
\label{MGS}
\displaystyle  k_1+ \dots +k_m+k_1q + \dots +k_mq+\dots +k_1q^{n-1} +\dots +k_mq^{n-1} + \dots,
\end{equation}
where $k_1, k_2, \dots k_m$ are fixed positive integers and  $q \in \left(0, 1\right)$.

Finally, notice that the achievement set for the sequence associated to the multigeometric series \eqref{gms} equals the arithmetic sum of the achievement sets associated to the different multigeometric series given by any partition of the scalars $k_1,\dots, k_m $. Therefore the results here can also be seen as criteria to determine the topological nature of the arithmetic sum of the achievement sets of the multigeometric sequences.

\section{The achievement set of some generalized multigeometric sequences}

Observe that the generalized multigeometric series given in (\ref{gms}), associated to a given function $f$ and to positive real numbers $k_1,\dots,k_m$,  can be written as
\begin{equation}\label{mgsw}
\sum_{n\ge 1} w_n(x) \quad \text{with}\quad w_n(x)=k_{g(n)}f(x^{\lfloor\frac{m+n-1}{m}\rfloor}),
\end{equation}
where, as usual, $\lfloor\,\cdot\,\rfloor$ denotes the integer part function and
$$g(n)=1+\big((n-1)\mod{m}\big).$$
Observe also that the achievement set for this sequence at a given $x$ is

$$E\bigl(w_n(x))\bigr)=\left\{ \sum_{n=1}^{\infty} \alpha_n f(x^n),\, \alpha_n \in A\right\}\quad\text{where}\quad  A=\left\{ \sum_{i=1}^{m} c_i k_i ,\, c_i \in \{ 0, 1 \} \right\}.$$

We will study the topology of this set for a particular class of functions:

\begin{definition}\label{def:HypoM}
    A function $f$ is {\em locally increasing  and power bounded (at $0$)}  if there exist $\epsilon\in(0,1)$, and $a,b,r\in\R^+$ such that $f$ is monotone increasing in $[0,\epsilon]$ and $$a\cdot x^r\leq f(x)\leq b\cdot x^r$$ for every $x\in[0,\epsilon].$ We denote by $\calm$ the class of locally increasing at 0 and power bounded functions.
\end{definition}
The following shows that differentiable functions abound in $\calm$:
\begin{proposition}\label{prop:HypoM_is_mild_hypothesis}
Let $f \in \mathcal{C}^{r+1}\big([0,1)\big)$ such that $f^{i)}(0)=0$ for $0\leq i<r$ and $f^{r)}(0)>0$. Then $f\in\calm$.
\end{proposition}
\begin{proof}
With $f$ as in the statement there exists $\epsilon\in(0,1)$ such that $f^{r)}([0,\epsilon])\subset\R^+$, and therefore $f(x)$ is monotone increasing in $[0,\epsilon]$.

On the other hand, define
$$
a=\frac{1}{r!}\min\big\{f^{r)}(\zeta),\,\zeta\in [0,\epsilon]\big\}\quad \text{ and }\quad b=\frac{1}{r!}\max\big\{f^{r)}(\zeta),\,\zeta\in [0,\epsilon]\big\},
$$
and consider, for any  $\lambda\in\R$,  the function $h_\lambda(x)=f(x)-\lambda x^r$. We then use the Taylor $(r-1)$th-approximation together with the error formula of $h_\lambda$ at $0$ to conclude that
\begin{equation}\label{eq:expression_h}
    h_\lambda(x)=\left(\frac{f^{r)}(\zeta)}{r!}-\lambda\right)x^{r}\quad\text{for some}\quad \zeta\in[0,x).
\end{equation}
In particular, for any $x\in [0,\epsilon]$,  $h_a(x)=f(x)-ax^r\geq 0$  while $h_b(x)f(x)-ax^r\leq 0$.
\end{proof}

\begin{example}
(1) Note that the identity $f(x)=x$ is trivially in $\calm$ by choosing $a=b=r=1$ and any $\epsilon$.

(2) By the well known Jordan inequality,
$$\frac{2x}{\pi} \leq \sin{x} \leq  x,\quad |x|\le \frac{\pi}{2},$$
we see that the function $f(x)=\sin x$ is in $\calm$ by choosing $a=\frac{2}{\pi}$, $b=1$ and $\epsilon=\frac{\pi}{2}$. Nevertheless, as $f'(x)=\cos x$, Proposition \ref{prop:HypoM_is_mild_hypothesis} provides $a=\cos 1$, $b=r=\epsilon=1$. The same applies, for instance, to the function $f(x)=\tan x$ choosing $r=\epsilon=1$:
 $$x \leq \tan x \leq \frac{x}{\cos^{2}(1)},\quad x\in[0,1].$$

(3) Consider the function $f(x)=x \cdot \ln (x+1)$ in which $f(0)=f'(0)=0$ and $f''(x)=\frac{x+2}{(x+1)^2}>0$ for $x \in [0, 1]$. Then $f\in\calm$ choosing  $r=2$, $\epsilon=1$,  $a=\frac{3}{8}$ and $b=1$.
Another example covered by Proposition \ref{prop:HypoM_is_mild_hypothesis} and providing $r=2$ is, for instance,
 $f(x)=e^x-x-1$. Here $\epsilon=1$, $a=\frac{1}{2}$ and $b=\frac{e}{2}$ .
\end{example}

We also need:

\begin{lemma} Let $f\in\calm$ and let  $k_1\ge k_2\ge \dots\ge k_m>0$ be positive scalars. Then, the associated generalized multigeometric series  $\sum_{n\ge 1}w_n(x)$ is convergent for any $x\in [0,\epsilon]$. Moreover,  $w_n(x)\ge w_{n+1}(x)$ for any $n\in \N$ and any $x\in [0,\epsilon]$.
\end{lemma}

\begin{proof} The second assertion trivially holds: as $f$ is monotone increasing in $[0,\epsilon]$ it follows  that, for any $x$ in this interval, $f(x^n)>f(x^{n+1})$. Hence, as $k_1\ge k_2\ge \dots\ge k_m$, it follows that $w_n(x)>w_{n+1}(x)$.

On the other hand, if we write $K=\sum_{i=1}^mk_i$, we deduce that
$$\sum_{n=1}^{\infty}w_n(x) \leq b \cdot \sum_{n=1}^{\infty} K x^{nr}=\frac{bKx^r}{1-x^r},\quad x\in [0,\epsilon],$$
and therefore, this series converges since $\epsilon\in (0,1)$.
\end{proof}

In what follows we fix an arbitrary function $f\in\calm$, thus constants $\epsilon\in(0,1]$, and $a,b,r\in\R^+$ are those given in Definition \ref{def:HypoM}. We choose positive real numbers $k_1\ge k_2\ge \dots\ge k_m>0$, and consider the associated generalized multigeometric series  $\sum_{n\ge 1}w_n(x)$ for $x\in[0,\epsilon)$. We also fix the following notation:
$$
K=\sum_{i=1}^{m} k_i,\quad U_j=\sum_{i=j+1}^{m} k_i,\quad L_j=\sum_{i=1}^{j} k_i=K-U_j,\quad
$$
Also, for any series $\sum_{n\ge 1} z_n$ and any $\ell\ge 1$ we denote by $Z_\ell=\sum_{n>\ell}z_n$ the $\ell$th {\em tail} of the series. In particular, we write $W_\ell(x)=\sum_{n> \ell} w_n(x)$.

On the other hand, we will strongly use the following foundational result of Kakeya, rediscovered and extended by Hornich:

\begin{theorem} \cite{Hornich,Kakeya}
\label{KHM}
Let $\sum_{n\ge1} z_n$ be a convergent positive series with non-increasing terms, i.e., $z_{n} \geq z_{n+1}$ for any $n \in \N$. Then, the achievement set $E(z_n)$ is:
\begin{enumerate}[label={\rm (\roman{*})}]
\item\label{thm:KHM.2} a finite union of bounded closed intervals if and only if $z_{n} \leq Z_{n}$
 for all but finitely many $n\in\N$;
\item\label{thm:KHM.3} a compact interval if and only if $z_n\leq Z_n$ for every $n \in \N$;
\item\label{thm:KHM.4} homeomorphic to the Cantor set if $z_{n} > Z_{n}$
 for all but finitely many $n\in\N$.
\end{enumerate}
\end{theorem}

With the notation above define 
$$
d_I=\sqrt[r]{\max_{1\leq j\leq m}\Big\{\frac{bk_j-aU_j}{bk_j+aL_j}\Big\}}.
$$
Our first result extends and refines \cite[Theorem 3.1]{dmytros}:
\begin{theorem}\label{thm:intervals}
Whenever $d_I\le \epsilon$, the achievement set $E\bigl(w_n(x)\bigr)$ is a compact interval for  any $x\in [ d_I,\epsilon]$.
\end{theorem}
\begin{proof}
According to the Theorem \ref{KHM}.\ref{thm:KHM.3}, it is enough to show that $w_n(x)\leq W_n(x)$ for every $n \in \N$ and  any $x\in [ d_I,\epsilon)$. Since $f\in\calm$,
$$
    w_n(x) =k_{g(n)}f(x^{\lfloor\frac{m+n-1}{m}\rfloor}) \leq b \cdot k_{g(n)} x^{\lfloor\frac{m+n-1}{m}\rfloor r},
$$
while
$$\label{eq:intervals2}
    W_n(x) =\sum_{\ell>n}  w_\ell(x) \geq a\cdot \sum_{\ell>n} k_{g(\ell)} x^{\lfloor\frac{m+\ell-1}{m}\rfloor r}=a\cdot x^{\lfloor\frac{m+n-1}{m}\rfloor r}\Big( U_{g(n)}  + K\frac{x^r}{1-x^r}\Big),
$$
for $n\in \N$ and $x\in[0,\epsilon]$.
Therefore $w_n(x)\leq W_n(x)$ whenever $$b\cdot k_{g(n)}x^{\lfloor\frac{m+n-1}{m}\rfloor r}\leq a\cdot x^{\lfloor\frac{m+n-1}{m}\rfloor r}\Big( U_{g(n)}  + K\frac{x^r}{1-x^r}\Big).$$
For $x=0$ this trivially holds and, for $x>0$ this is the case when
 $$b\cdot k_{g(n)}\leq a\cdot \Big( U_{g(n)}  + K\frac{x^r}{1-x^r}\Big),$$
that is, whenever
$$x\geq \sqrt[r]{\frac{bk_{g(n)}-aU_{g(n)}}{bk_{g(n)}+aL_{g(n)}}}$$
for all $n$. As  the function $g(n)$ takes the values $1, 2, \dots, m$, this happens for all $x\in [ d_I,\epsilon)$.
\end{proof}
As a consequence, denoting
$$
d_{IM} := \sqrt[r]{\frac{b}{b+a}},
$$ we obtain:
\begin{corollary}
\label{cor:contains_a_interval}
Whenever $d_{IM}\le \epsilon$, the achievement set  $E\bigl(w_n(x)\bigr)$ is a compact interval for any
$x \in [ d_{IM},\epsilon ]$.
\end{corollary}
\begin{proof}
Simply note that
$$ d_{I}=\sqrt[r]{\max_{1\leq j\leq m}\Big\{\frac{bk_j-aU_j}{bk_j+aL_j}\Big\}} =  \sqrt[r]{\max_{1\leq j\leq m}\Big\{\frac{b-aU_j/k_j}{b+aL_j/k_j}\Big\}} \leq \sqrt[r]{\frac{b}{a+b}}=d_{IM}$$
and apply Theorem \ref{thm:intervals}.
\end{proof}

On the other hand, denoting
$$d_{NI} = \sqrt[r]{\frac{ak_m}{bK+ak_m}},$$
we prove the following that extends \cite[Theorem 3.3]{dmytros}, which in turn is inspired by \cite[Theorem 2.1]{BFS} generalized in  \cite[Theorem 2.2(ii)]{theFerdinands}:

\begin{theorem}\label{thm:non_fin_int}
The achievement set  $E\bigl(w_n(x)\bigr)$ is not a finite union of closed bounded intervals for $0< x< \min\{\epsilon, d_{NI}\}$.
\end{theorem}

\begin{proof}

According to Theorem \ref{KHM}.\ref{thm:KHM.2} it is sufficient to show that $w_{\ell m}(x)>W_{\ell m}(x)$ for any $\ell  \in \N$ and $0< x< \min\{\epsilon, d_{NI}\}$. Observe that, for any $x\in(0,\epsilon)$,
$$w_{\ell m}(x)=k_m f(x^\ell) \geq a \cdot k_m x^{\ell r},$$
while
$$W_{\ell m}(x)=K \cdot \sum_{j>\ell} f(x^j) \leq b \cdot \frac{Kx^{(\ell+1)r}}{1-x^r}.$$
Therefore $w_{\ell m}(x) > W_{\ell m}(x)$ as long as

$$W_{{\ell}m}\leq b \cdot \frac{Kx^{({\ell}+1)r}}{1-x^r} < a \cdot k_m x^{{\ell}r} \leq w_{{\ell}m}.$$
That is, whenever
$$x<\sqrt[r]{\frac{a k_m}{bK+ak_m}}=d_{NI}.$$

\end{proof}

The next result extends \cite[Theorem 2.1]{BFS}, \cite[Theorem 2.2(i)]{theFerdinands} and \cite[Theorem 3.4]{dmytros}. We follow  a strategy similar to the proofs of these references. 

\begin{theorem}\label{thm:arithmetic_contains_interval}
Choose $\lambda,\mu\in\R^+$ and  $s\in\N$ such that every number ${\mu}, {\mu}+\lambda, {\mu}+2\lambda, \dots, {\mu}+s\lambda$,
%can be obtained by summing up the numbers $k_1, k_2, \dots, k_m$,
is a subsum of the (finite) series $\sum_{i=1}^m k_i$, and write  
$$d_{CI} = \sqrt[r]{\frac{b}{s \cdot a + b}}\,.$$
Then, whenever $d_{CI} < \epsilon$, $E\bigl(w_n(x)\bigr)$ contains a compact interval for any $x\in [ d_{CI},\epsilon)$.
\end{theorem}
\begin{proof}
Let define
$\overline{k}_1=\overline{k}_2=\ldots=\overline{k}_{s}=\lambda>0$, and consider the convergent positive series $\sum_{n\ge 1} \overline w_n(x)$ where
$$
\overline{w}_n(x)=\overline{k}_{\overline{g}(n)}f(x^{\lfloor\frac{n+s-1}{s}\rfloor}),
$$
with $\overline{g}(n)=1+\big((n-1)\mod{s}\big).$
Then, according to Theorem \ref{thm:intervals}, $E\bigl(\overline{w}_n\bigr)$ is a compact interval for all $x\in[d_I,\epsilon)$ where now
\begin{align*}
 d_I&=\sqrt[r]{ \max_{1\leq j\leq s}\Big\{\frac{b\overline{k}_j-aU_j}{b\overline{k}_j+aL_j}\Big\}}\\
 &=\sqrt[r]{ \max_{1\leq j\leq s}\Big\{\frac{b\lambda-a\lambda(s-j)}{b\lambda+a\lambda j}\Big\}}\\
&=\sqrt[r]{\max_{1\leq j\leq s}\Big\{1-\frac{as}{aj+b}\Big\}}\\
&=\sqrt[r]{1-\frac{as}{as+b}}\\
&=\sqrt[r]{\frac{b}{b+sa}}\,.
\end{align*}
We finish the proof by showing that the interval
$$\{\sum_{n=1}^\infty {\mu}f(x^n)\}+E\bigl(\overline{w}_n(x)\bigr)
$$
is contained in $E\bigl(w_n(x)\bigr)$.

Indeed, if $z\in \{\sum_{n=1}^\infty {\mu}f(x^n)\}+E\bigl(\overline{w}_n(x)\bigr)$,  write
$$z=\sum_{n=1}^\infty \big({\mu}+ s_n\lambda)f(x^n),$$
where $s_n\in\{0,1,\ldots,s\}$. Therefore, there exist $c_{n,i}\in\{0,1\}$ such that ${\mu}+s_n\lambda=\sum_{i=1}^{m}c_{n,i}k_i$, and thus
$$z=\sum_{n=1}^\infty \big(\sum_{i=1}^{m}c_{n,i}k_i\big)f(x^n)\in E(w_n).$$
\end{proof}
The following recovers \cite[Theorem 2.2(iii)]{theFerdinands} and generalizes \cite[Corollary 3.5]{dmytros}. Under the hypothesis of the previous theorem we have:

\begin{corollary}\label{cor:Cantorval_ferdinands}
Whenever $d_{CI}<d_{NI}$, the achievement set $E\bigl(w_n(x)\bigr)$ is a Cantorval for any $x\in [ d_{CI},d_{NI})$.
\end{corollary}
\begin{proof}
Let $x \in [ d_{CI}, d_{NI})$. Then
\begin{enumerate}[label={\rm (\roman{*})}]
    \item According to the Theorem \ref{thm:arithmetic_contains_interval}, if $x \geq \sqrt[r]{\frac{b}{sa+b}}$, then $E\bigl(w_n(x)\bigr)$ contains an interval.

    \item An analogous argument to the one in the proof of Theorem \ref{thm:non_fin_int} shows that if $x < \sqrt[r]{\frac{ak_{m}}{ak_{m}+bK}}$, then $w_{mi}(x)> W_{mi}(x)$ for every $i \in \N$ and $E(w_n)$  cannot be a finite union of closed and bounded intervals.
\end{enumerate}
Therefore, $E\bigl(w_n(x)\bigr)$ must be a Cantorval.
\end{proof}

\begin{remark} Observe that $d_{CI}< d_{NI}$ only if
$$b<a \sqrt{\frac{s  k_m}{K}}.$$
Therefore, as $a\leq b$, $$1<\frac{s  k_m}{K}$$ is a necessary condition for Corollary \ref{cor:Cantorval_ferdinands}.
\end{remark}

Finally, let 
$$
d_C=\sqrt[r]{\min_{1\leq j\leq m}\Big\{\frac{ak_j-bU_j}{ak_j+bL_j}\Big\}.}
$$
Our last result extends and refines \cite[Theorem 3.7]{dmytros}:
\begin{theorem}\label{thm:cantor}
Whenever $d_{C}>0$, the achievement set $E\bigl(w_n(x)\bigr)$ is homeomorphic to the Cantor set for $0< x< \min\{\epsilon, d_{C}\}$.
\end{theorem}
\begin{proof}
According to Theorem \ref{KHM}.\ref{thm:KHM.2}, it is enough to show that $w_n(x)> W_n(x)$ for all but finitely many $n\in\N$ and $x\in[0,d_C]$. Observe that, for any $x\in(0,\epsilon)$,

\begin{align*}
    w_n(x) =k_{g(n)}f(x^{\lfloor\frac{m+n-1}{m}\rfloor}) \geq a\cdot k_{g(n)}\big( x^{\lfloor\frac{m+n-1}{m} \rfloor r} \big)
\end{align*}
while
\begin{align*}
    W_n(x)
    &=\sum_{\ell>n} w_\ell
    \leq b\cdot  \sum_{\ell>n} k_{g(\ell)}\big(x^{\lfloor\frac{m+\ell-1}{m}\rfloor r}\big)
    =b\cdot x^{\lfloor\frac{m+n-1}{m}\rfloor r}\Big( U_{g(n)}  + \frac{Kx^r}{1-x^r}\Big).
\end{align*}

Therefore, $w_n(x)> W_n(x)$ as long as
$$a\cdot k_{g(n)}x^{\lfloor\frac{m+n-1}{m}\rfloor r}>b\cdot x^{\lfloor\frac{m+n-1}{m}\rfloor r}\Big( U_{g(n)}  + \frac{Kx^r}{1-x^r}\Big),$$
and since $x>0$, whenever
$$a\cdot k_{g(n)}>b\cdot \Big( U_{g(n)}  + \frac{Kx^r}{1-x^r}\Big).$$
That is, when
$$x<\sqrt[r]{\frac{ak_{g(n)}-bU_{g(n)}}{ak_{g(n)}+bL_{g(n)}}}$$
for all $n$ and the theorem follows.
\end{proof}

\bibliographystyle{abbrv}
\bibliography{Cartovals_ref}
\end{document}